\documentclass[12pt, twoside, leqno]{article}

\usepackage{amsmath,amsthm}
\usepackage{amssymb}

\usepackage{enumerate}

\usepackage{graphicx}

\newtheorem{thm}{Theorem}[section]

\newtheorem{prop}[thm]{Proposition}

\theoremstyle{definition}
\newtheorem{defin}[thm]{Definition}
\newtheorem{rem}[thm]{Remark}

\numberwithin{equation}{section}

\graphicspath{{./pictures/}}

\newcommand{\R}{\mathbb{R}}

\begin{document}




\title{\bf{Involutions of real intervals}}

\author{Gaetano Zampieri\\
Universit\`a di Verona\\ 
Dipartimento di Informatica\\
strada Le Grazie 15, 37134 Verona, Italy\\
E-mail: gaetano.zampieri@univr.it
}

\date{}

\maketitle

\begin{center}
\emph{Dedicated to Jorge Sotomayor for his $70^{th}$ birthday}
\end{center}


\renewcommand{\thefootnote}{}
\footnote{To appear  in \emph{Annales Polonici Mathematici}.}

\renewcommand{\thefootnote}{\arabic{footnote}}
\setcounter{footnote}{0}


\begin{abstract}
This paper shows  a simple construction of the continuous involutions of  real intervals in terms
of the continuous even functions. We also study the smooth involutions defined by symmetric equations.
Finally, we review some applications, in particular the characterization of the isochronous potentials by means
of  smooth involutions.\end{abstract}

\section{Introduction}
An  \emph{involution} is a function  that is its own inverse.
This is an important object in all mathematical fields. We are going to consider continuous involutions
on real  intervals  only. 

\begin{prop}\label{P1}
Let $h:J\to J$ be continuous function on the interval $J\subseteq \R$  which is the inverse of itself and does not coincide with the  identity function $id_J$. Then $h$ is strictly decreasing and has a unique fixed point $\bar x=h(\bar x)$.
\end{prop}
\begin{proof} The function $h$ is  strictly monotonic being continuous and injective on an interval. Let us prove that   $h$  strictly decreases. Suppose it does 
 not,  then it  is increasing. Since $h\ne id_J$ then $h(x_0)\ne x_0$ for some $x_0\in J$. If $x_0<h(x_0)$ 
then $h(x_0)<h(h(x_0))=x_0$ a contradiction; similarly, $x_0>h(x_0)$ implies $h(x_0)>x_0$.
Thus $h$ decreases and 
the function $k(x)=x-h(x)$ strictly
increases. Fixed points of $h$ coincide with zeros of $k$ and there is one zero at most since $k$ strictly
increases. Consider a point $x_1\in J$. If $k(x_1)=0$ then $x_1$ is the unique fixed point of $h$. If $k(x_1)>0$, then $x_1>h(x_1)$ and $k(h(x_1))=h(x_1)-h(h(x_1))=h(x_1)-x_1=-k(x_1)<0$,
so, by the continuity of $k$,  there exists  $\bar x\in(h(x_1),x_1)$ such that $k(\bar x)=0$, namely $\bar x$ is the fixed point of $h$.
 If $k(x_1)<0$ we can argue similarly.
\end{proof}
These and other general properties are well known. Involutions are solutions
to the celebrated Babbage functional equation $\phi^n(x)=x$,  in the case $n=2$,
see the book \cite[Chap. XV]{K}, by Kuczma,  in particular Thms. 15.3 and 15.2, and Lemma 15.1.
See also 
Kuczma, Choczewski and  Ger \cite[Chap. 11]{KCG}, and Section~1 of~\cite[Chap. VIII]{PR} by Przeworska-Rolewicz
where $h$ as above is called a Carleman function.

For $\bar x$ as above, the function $x\mapsto h(x+\bar x)-\bar x$ is also an involution and has $0$ as fixed point, conversely
$x\mapsto h(x-\bar x)+\bar x$ has fixed point $\bar x$ if $h(0)=0$.

In the sequel we shall consider non-trivial involutions $h\ne id_J$, with $J$ open interval and $h(0)=0$, moreover we are going 
to study  \emph{smooth} $h\in C^1$ involutions. By the chain rule we have $h'(h(x))h'(x)=1$ so $h'(x)\ne 0$ at all $x$, more precisely $h'(x)<0$ since we excluded the identity, and $h'(0)=-1$ necessarily. So, the present paper  uses
the following terminology: 

\begin{defin}\label{involution} 
A continuous function $h$ 
 of an  open interval $J\subseteq \R$ onto itself is called an \emph{involution} if
\begin{equation}\label{h-properties}
  h^{-1}=h,\qquad 0\in J,\qquad
  h(0)=0,\qquad h\ne id_J.
  \end{equation}
  In particular it is called a smooth involution if $h\in C^1$, so it is a $C^1$ diffeomorphism
  with $h'(0)=-1$.
\end{defin}

Of course $h(x)=-x$ is an involution on the whole $\R$. The following piecewise-linear example is taken from \cite{PR} p.~177
\begin{equation}\label{piecewiselinear}
h(x)=\begin{cases} -x/\lambda, \quad x\le 0\\ -\lambda\, x, \ \quad x> 0\end{cases}\qquad x\in\R, \quad \lambda>0.
\end{equation}

A very simple smooth involution which seems to be ``new'' is 
\begin{equation}\label{logexp}
h(x)=\ln\bigl(2-e^x\bigr),\qquad x<\ln 2.
\end{equation}


\section{Constructing continuous involutions}\label{constructingcontinuosinvolutions}

Our main result  is the following:

\begin{thm}\label{theorhcontinuos}
Let $h:J\to J$ be a (continuous) involution as in Definition~\ref{involution}, then $k(x):=x-h(x)$  is a homeomorphism $J\to I$ with $I$ symmetric open interval,
and the function $P:I\to\R$ defined by $P(y)=2k^{-1}(y)-y$ satisfies $P(0)=0$ and is even.
Vice versa, if $P: I\to\R$, with $P(0)=0$, is a continuous even function on a symmetric open interval such that the function $K: I\to J$,
\begin{equation}\label{condcont}
 K(y)= \frac12 \bigl(y+P(y)\bigr),
\end{equation}
 is a  homeomorphism onto some $J$, then $h(x):=x-k(x)$, $k=K^{-1}$,
 is an involution on $J$.
\end{thm}

\begin{proof} If $h:J\to J$ is an involution  then $k(x):=x-h(x)$ is strictly increasing as we already
saw in the proof of Proposition~\ref{P1}, so a homeomorphism
onto some open interval $I$ as well known. The interval $I$ is symmetric
since 
\begin{equation*}
y=k(x)=x-h(x)\in I\Longrightarrow -y=-k(x)=k\bigl(h(x)\bigr)\in I.
\end{equation*}
Next, $h(x)=k^{-1}(k(h(x)))=k^{-1}(-k(x))$ and
\begin{multline*}
P(y):=2k^{-1}(y)-y=-2y+2k^{-1}(y)+y=-2k(k^{-1}(y))+2k^{-1}(y)+y=\\=-2k^{-1}(y)+2h(k^{-1}(y))+2k^{-1}(y)+y=
2k^{-1}(-y)+y=P(-y).
\end{multline*}
This fact and $P(0)=0$ prove the first sentence. To prove the second sentence let us plug $k(x)=K^{-1}(x)$ into \eqref{condcont}, and then plug $-k(x)$
\begin{multline*}
x=\frac12 \bigl(k(x)+P(k(x))\bigr), \quad k^{-1}\bigl(-k(x)\bigr)=\frac12 \bigl(-k(x)+P(-k(x))\bigr)=\\
=\frac12 \bigl(-k(x)+P(k(x))\bigr) \Longrightarrow x-k^{-1}\bigl(-k(x)\bigr)=k(x)=x-h(x).
\end{multline*}
Thus $h(x)=k^{-1}(-k(x))$ which shows that $h$ is a homeomorphism,  
$h\ne id_J$ (otherwise $k=-k$ so $k\equiv 0$), and $h(h(x))=x$ for all $x\in J$ namely $h^{-1}=h$. Finally $P(0)=0$ implies $k(0)=0$ and $h(0)=0$.
\end{proof}

Since $I=k(J)$ we have that $I=(\inf J-\sup J, \sup J-\inf J)$ when both $\inf J,\sup J\in \R$,
and $I=\R$ otherwise. 

Of course, if we consider an \emph{arbitrary} $C^1$ even function $P$, with $P(0)=0$, then $P'(0)=0$, and formula \eqref{condcont} restricted to the maximal symmetric open interval $I$ where $K'(y)>0$, defines a strictly increasing diffeomorphism onto an 
interval $J$, and $h(x):=x-k(x)$, with $k=K^{-1}$,
 is a smooth involution on $J$.

\begin{figure}
\begin{center}
\includegraphics[width=.4\textwidth]{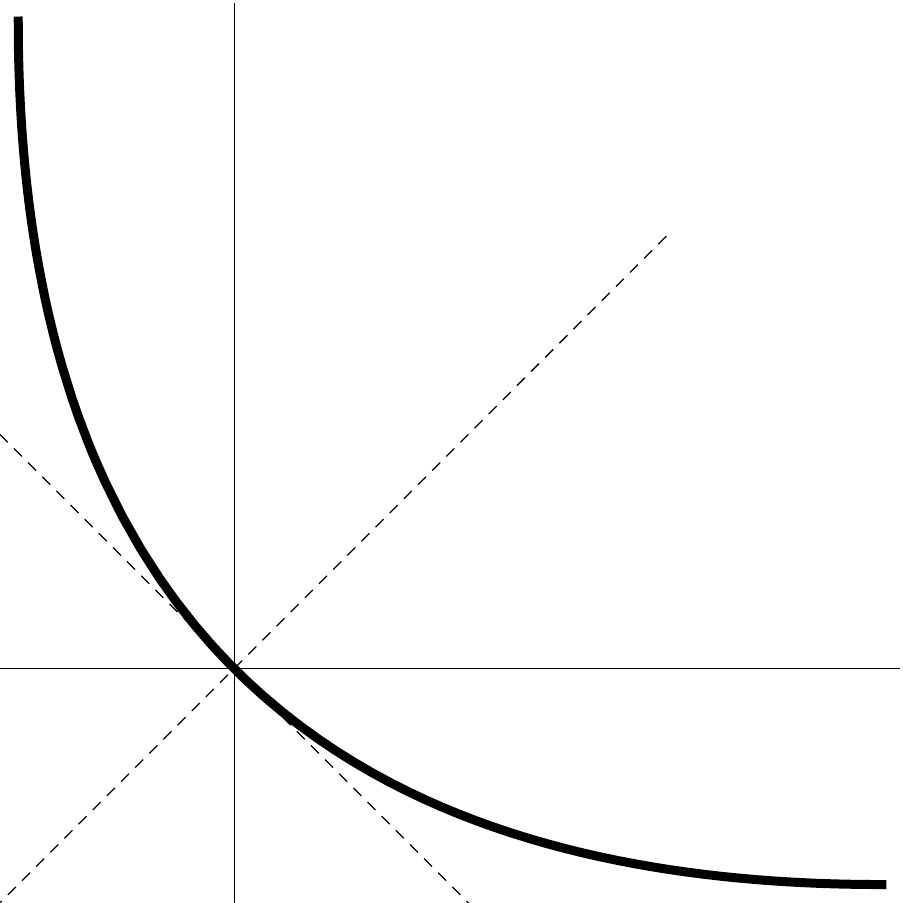}
\end{center}
\caption{Involution \eqref{parabolic-h} given by $P(y)=y^2/8$.}
\label{parabolic}
\end{figure}
For instance, starting from $P(y)=y^2/8$, $y\in \R$, formula \eqref{condcont} defines $K(y)=y/2+y^2/16$ and $K'(y)>0$ if and only if $y>-4$. So $K$ is injective on
the symmetric open interval $I=(-4,4)$ and a homeomorphism  $I\to J=K(I)=(-1,3)$. The function $k(x)=K^{-1}(x)=-4+4\sqrt{1+x}$ and finally we get the involution
$h(x)=x-k(x)$ namely
\begin{equation}\label{parabolic-h}
h:(-1,3)\to (-1,3),\; x\mapsto x+4-4\sqrt{1+x}.
\end{equation}

If we start from $P(y)=y^6$ we have  
\begin{equation*}
k:J=\bigl(\textstyle{\frac{-5}{12\cdot 6^{1/5}},\frac{7}{12\cdot 6^{1/5}}\bigr)\to I=\bigl(\frac{-1}{6^{1/5}},\frac{1}{6^{1/5}}}\bigr)
\end{equation*}
and $h:J\to J$ is the non elementary algebraic function in Figure~\ref{algebraic}.
 \begin{figure}
\begin{center}
\includegraphics[width=.4\textwidth]{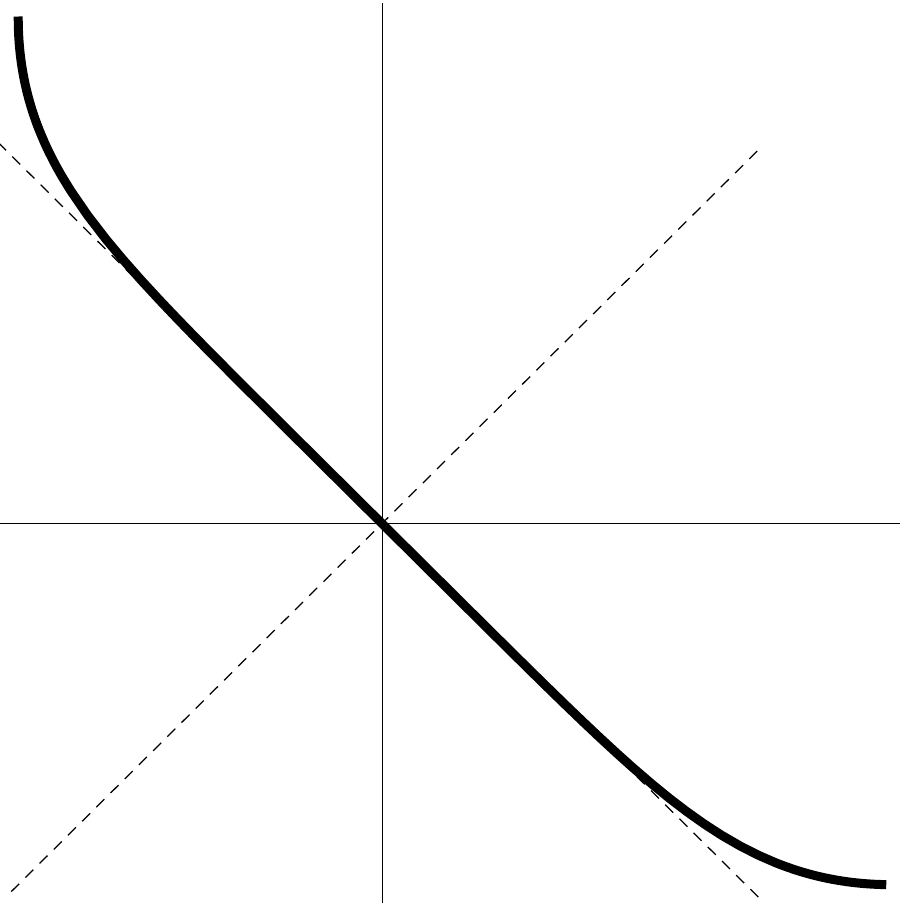}
\end{center}
\caption{Involution  given by $P(y)=y^6$.}
\label{algebraic}
\end{figure}

To illustrate the first part of the theorem, consider the piecewise linear involution \eqref{piecewiselinear}. 
It gives
\begin{equation*}
k(x)=\begin{cases} (1+\lambda) x/\lambda, \  x\le 0\\ (1+\lambda)\,x,\; \quad x>0\end{cases}\quad
k^{-1}(y)=\begin{cases} \lambda\, y/(1+\lambda), \  y\le 0\\ y/(1+\lambda), \quad y>0\end{cases}
\end{equation*}
and the even function $P(y)=2 k^{-1}(y)-y$ is
\begin{equation}
P(y)=\frac{1-\lambda}{1+\lambda}\;|y|,\qquad y\in\R,\quad \lambda>0.
\end{equation}

Finally,  the smooth involution \eqref{logexp} gives
\begin{equation*}
k(x)=x-\ln\bigl(2-e^x\bigr),\quad k\bigl(-\infty,\ln 2)=\R,\quad k^{-1}(y)=\ln \frac{2}{1+e^{-y}},
\end{equation*}
 and the following even function  
\begin{equation}
P(y)=-y+2 \ln \frac{2}{1+e^{-y}}=-2\ln \cosh \frac{y}{2},\qquad y\in\R.
\end{equation}
 
 Functional equations relating involutory and even functions are studied in Schwerdtfeger \cite{S},
however our simple Theorem~\ref{theorhcontinuos} seems to be new.
\section{Involutions given by  symmetric equations}\label{symmetric}

The condition $h^{-1}=h$ is equivalent to the symmetry of the graph of~$h$  with respect to the diagonal; indeed  $(x, h(x))$ has $(h(x),x)$ as symmetric point and this coincides
with the point $(h(x),\allowbreak h(h(x)))$ of the graph.   For example, consider the hyperbola $y\,x=1$, which is symmetric with respect to the diagonal. In order to fulfill the further condition $h(0)=0$,   we translate its point $(1,1)$ to the origin. In this way
 we get $(y+1)(x+1)=1$, which can be solved for $y$ as $y=-x/(1+x)$. If we finally take the branch that goes through the origin we arrive at the following involution
\begin{equation}\label{h-upperrectangularhyperbola}
  h(x)=-\frac{x}{1+x},\qquad x\in J=(-1,+\infty)\,.
\end{equation}

Involutions are preserved by  homothety:

\begin{rem}\label{remarkfamily} Let $a\in\R\setminus\{0\}$ and $h$ be an involution on $(b,c)$, then
 $\tilde h(x)=h(a\,x)/a$ is an involution on $(b/a,c/a)$ if $a>0$, on $(c/a,b/a)$ otherwise.
\end{rem}
In this way \eqref{h-upperrectangularhyperbola} gives the following 1-parameter family of involutions
\begin{equation}\label{h-rational}
  h(x)=-\frac{x}{1+a x}\,, \qquad 
  x\in J=\begin{cases}
  (-1/a,+\infty),&a>0\\
  (-\infty,+\infty),&  a=0\\
  (-\infty,-1/a),& a<0\end{cases}
\end{equation}
These are the only involutions that are rational functions of~$x$ as shown in Cima,  Ma\~nosas, and Villadelprat~\cite{CMV}.

Acz\'el in \cite{A} and Shisha and Mehr in \cite{SM} obtain injective functions $\hat h:\R\to\R$ such that $\hat h^{-1}=\hat h$ from symmetric
 functions $f:\R^2\to\R$,
$f(x,y)=f(y,x)$. The paper \cite{SM} suppose that for every $x\in\R$ there exists a unique $y$ to be denoted by $\hat h(x)$ such that
$f(x,y)=0$, then $\hat h$ satisfies $\hat h^{-1}=\hat h$. In particular, $f(x,y)=x^3+y^3-a$ gives  $\hat h(x)=\sqrt[3]{a-x^3}$ which has the fixed point $\bar x=\sqrt[3]{a/2}$. The function $x\mapsto \hat h(x+\bar x)-\bar x$, i.e.
\begin{equation}\label{SMGCI}
\R\to\R,\quad x\mapsto \sqrt[3]{a-\left(x+\sqrt[3]{a/2}\right)^3}-\sqrt[3]{a/2},
\end{equation}
 is an involution
in the sense of Definition~\ref{involution}. For $a\ne 0$ it  is  non-differentia\-ble at $x=\sqrt[3]{a}-\sqrt[3]{a/2}$. 
 \begin{figure}
\begin{center}
\includegraphics[width=.5\textwidth]{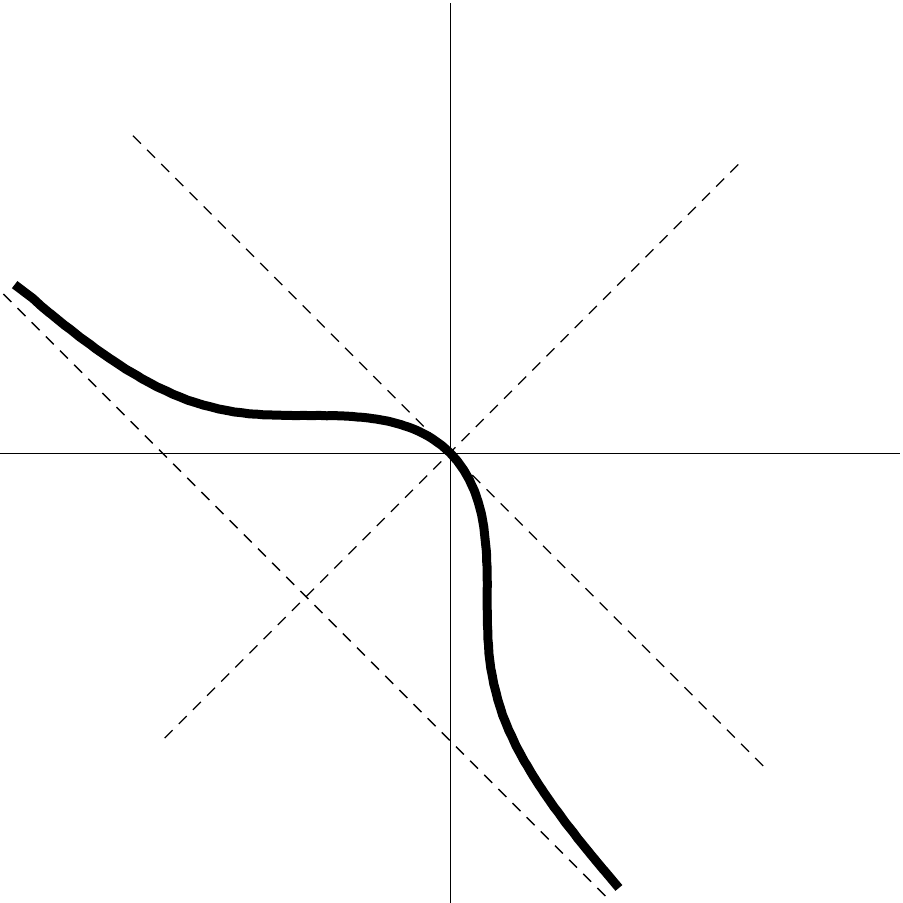}
\end{center}
\caption{Global involution \eqref{SMGCI} for $a=2$.}
\label{shisha}
\end{figure}

We consider smooth symmetric equations in order to use the implicit function theorem:
\begin{prop}\label{smoothinvolution}
Let $f:\Omega\to\R$ be a $C^1$ function on the open set $\Omega \subseteq\R^2$ such that: $(0,0)\in\Omega$, $f(0,0)=0$, and
\begin{equation}\label{symmetry2}
  (x,y)\in \Omega\quad\Longrightarrow\quad
  (y,x)\in\Omega,\quad f(y,x)=f(x,y).
\end{equation}
Let $\Gamma$ be the connected component of $f^{-1}(0)$ that contains the origin. Suppose that $\partial_2f(x,y)\ne 0$ for all $(x,y)\in\Gamma$. Then $\Gamma$ is the graph of a smooth involution $h$. All smooth involutions can be obtained this way.
\end{prop}
\begin{proof} By the implicit function theorem, $\Gamma$~is the graph of a $C^1$ function $h$ with $h(0)=0$. 
Let $J$ be the projection of $\Gamma$ on the $x$-axis. It is an open interval since $\Gamma$ is open and connected. From~\eqref{symmetry2}
we have $\partial_1f(x,y)=\partial_2f(y,x)$ for $(x,y)\in\Omega$ so $\partial_1f$ never vanishes on $\Gamma$ and has the same sign as $\partial_2f$. We deduce that
\begin{equation*}
  h'(x)=-\frac{\partial_1f\bigl(x,h(x)\bigr)}
  {\partial_2f\bigl(x,h(x)\bigr)}<0,\qquad x\in J,
\end{equation*}
and in particular $h'(0)=-1$. Finally, $f(h(y),y)=f(y,h(y))=0$ shows that $h^{-1}=h$.

Let us prove the last sentence. Let $h:J\to J$ be a smooth involution and define $f(x,y):= x+y-h(x)-h(y)$. This is a $C^1$ function on $J\times J$, $f(0,0)=0$, and $f(y,x)=f(x,y)$. We have $\partial_2f(x,y)=1-h'(y)>0$ for all $(x,y)\in J\times J$. The graph of $h$  coincides with   $f^{-1}(0)$. Indeed,
if $y=h(x)$ then $x=h(y)$ and $f(x,y)=0$; conversely, if $f(x,y)=0$ then $k(y)=y-h(y)=-(x-h(x))=-k(x)$,  so $y=k^{-1}(k(y))=k^{-1}(-k(x))=k^{-1}(k(h(x))=h(x)$.
\end{proof}

 \begin{figure}
\begin{center}
\includegraphics[width=.5\textwidth]{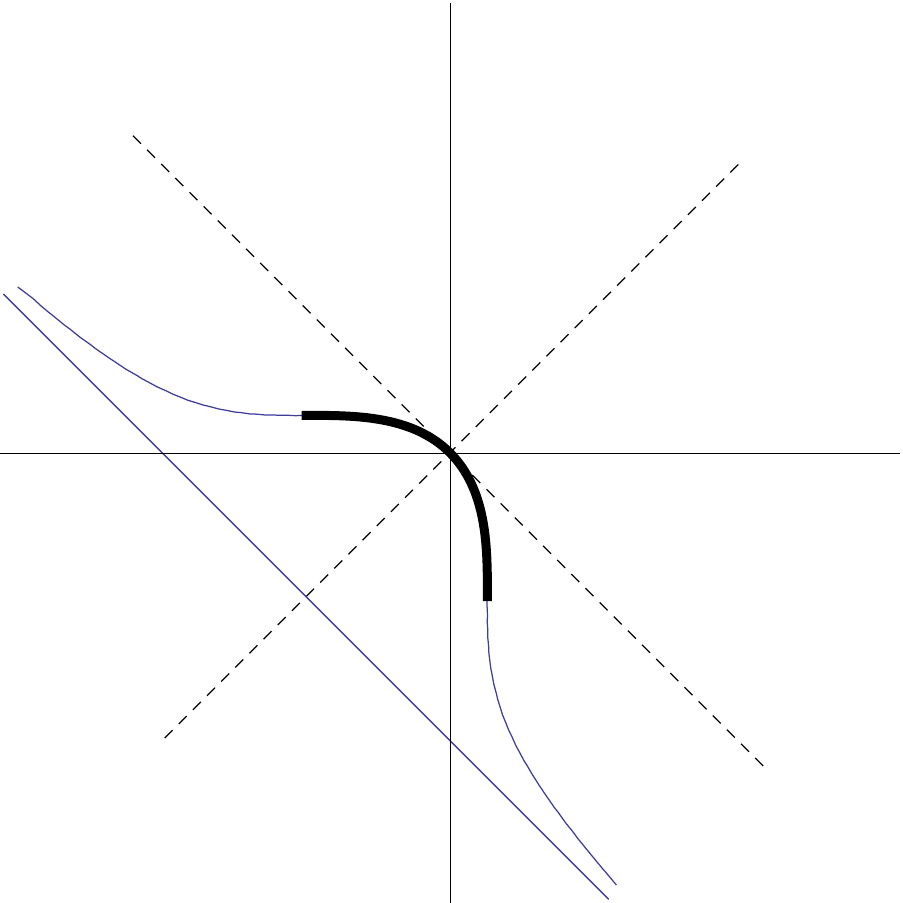}
\end{center}
\caption{Smooth involution in \eqref{h-shishasmooth}.}
\label{shishasmooth}
\end{figure}
For instance, let us consider the following function on the whole $\R^2$
\begin{equation*}
\tilde f(x,y)=\left(\bigl(x+1\bigr)^3+\bigl(y+1\bigr)^3-2\right)\left(x+y+2\right),
\end{equation*}
let $\tilde \Gamma$ be the cubic plane curve $\bigl(x+1)^3+\bigl(y+1)^3-2=0$ and let $L$ be the straight line $x+y+2=0$.
The connected sets $\tilde \Gamma$ and $L$ are disjoint, moreover $(0,0)\in\tilde\Gamma$. We have 
$\partial_2\tilde f(x,y)=0$ for $(x,y)\in\tilde \Gamma$ when $y=-1$  namely at
the point $(c,-1)\in\tilde\Gamma$ with $c=\sqrt[3]{2}-1$. So we define $\Omega=(-1,c)\times (-1,c)$, an
open square with $(0,0)\in\Omega$, and we have that the restriction $f=\tilde f|\Omega$, satisfies \eqref{symmetry2}.
The connected component of $f^{-1}(0)$ that contains the origin is $\Gamma=\tilde\Gamma\cap \Omega$ and $\partial_2f(x,y)\ne 0$ for all $(x,y)\in\Gamma$. Therefore $\Gamma$ is the graph of a smooth involution $h$. In this case
we can even write the explicit formula which is the following restriction of the   involution \eqref{SMGCI} for $a=2$
 \begin{equation}\label{h-shishasmooth}
h: \bigl(-1,\sqrt[3]{2}-1\bigr) \to \bigl(-1,\sqrt[3]{2}-1\bigr), x\mapsto \sqrt[3]{2-(x+1)^3}-1.
\end{equation}
The thick curve in Figure~\ref{shishasmooth} is the graph of this smooth involution which is a piece of the non-smooth graph $\tilde \Gamma$, compare with  Figure~\ref{shisha}. The straight line below $\tilde \Gamma$ is $L$.

\section{Isochronous potentials by involutions}\label{isochrony}

An equilibrium point of a planar vector field is called a (local) \emph{center} if all orbits in a neighborhood are
periodic and enclose it. The center is \emph{isochronous} if all periodic orbits have the same period.
The smooth involutions can be used to construct the isochronous centers for the scalar equation $\ddot{x}=-g(x)$
as proved  in the 1989 paper~\cite{Z2}, by  the present author. There are other different approaches to such isochronous centers, which do not
involve involutions,  in particular the 1961 Urabe's paper~\cite{U1}, see  also~\cite{U2}.

\begin{thm}\label{isochronous}
Let $h:J\to J$ be a smooth involution, $\omega>0$, and define 
\begin{equation}\label{Visochronous}
  V(x)=\frac{\omega^2}{8}\,\big(x-h(x)\big)^2,\qquad
  x\in J.
\end{equation}
Then  the origin is an isochronous center for $\ddot{x}=-g(x)$, where $g(x)=V'(x)$, namely all orbits which intersect the $J$ interval of the $x$-axis in the $x,\dot x$-plane, are periodic and have the same period $2 \pi/\omega$. Vice versa, let $g$ be continuous on a neighborhood of $0\in\R$, $g(0)=0$, suppose there exists $g'(0)>0$, and the origin is an isochronous center for $\ddot x=-g(x)$,  then there exist an open interval $J$, $0\in J$, which is a subset of the domain of $g$, and an involution $h:J\to J$  such that \eqref{Visochronous} holds with $V(x)=\int_0^xg(s)ds$ and $\omega= \sqrt{g'(0)}$.
\end{thm}

The potential $V$ of an isochronous center is called an \emph{isochronous potential}.
The proof is included in the proof of Proposition~1 in~\cite{Z2}  as a particular case. Formula \eqref{Visochronous} corresponds to formula~(6.2) in the paper~\cite{Z2}. A detailed proof can be also found in the recent~\cite{Z3}; see Theorem~2.1 and Corollary~2.2 in~\cite{Z3}. This last paper also contains 
the following  necessary conditions for a smooth enough potential $V$ to be isochronous:
\begin{equation}
  V^{(4)}(0)=\frac{5V'''(0)^2}{3V''(0)},\quad
  V^{(6)}(0)=\frac{7V'''(0)V^{(5)}(0)}{V''(0)}-
  \frac{140V'''(0)^4}{9V''(0)^3},
\end{equation}
which can be deduced by taking successive derivatives of the involution relation $h(h(x))\equiv x$ at $x=0$.
We can consider the necessary condition at any even order derivative, provided that $V$ admits that derivative.

Inserting the involution~\eqref{h-rational} into formula~\eqref{Visochronous} we obtain the following isochronous potential
\begin{equation}\label{V-upperrectangularhyperbola}
  V(x)=\frac{\omega^2}{8}\,x^2
  \Bigl(\frac{2+a x}{1+a x}\Bigr)^2,\quad
  x\in J=\begin{cases} (-1/a,+\infty),&a>0\\
  (-\infty,+\infty),&a=0\\
  (-\infty,-1/a),&a<0.\end{cases}
\end{equation}
This is the only isochronous rational potential  as proved in~\cite{CV}. 

The paper~\cite{GZ}, by Gorni and the present author, studies the global isochronous potentials $V:\R\to\R$ in terms
of smooth involutions. In particular it gives implicit examples and
new explicit ones. Also, the paper~\cite{GZ} revisits Stillinger and Dorignac global isochronous potentials in terms of  involutions which are 
given by  hyperbolas in Stillinger's case.

\section{Instability under some attractive central forces}\label{centralforces}
The paper \cite{Z1} considers the differential system
\begin{equation}\label{cf}
\ddot x=-xf(x),\qquad \ddot y=-y f(x),\qquad f(0)=1,
\end{equation}
where $f$ is continuous near $0$. It represents the motion under a particular \emph{attractive central force} which
is not a gradient. The origin of $\R^2$ is a (local) center for the first equation $\ddot x=-xf(x)$. Let us introduce the potential $V(x)=\int_0^x sf(s)ds$. For a suitable  open interval $J\ni 0$, the potential
$V$ is strictly increasing on $J\cap \R_+$, strictly decreasing on $J\cap \R_-$,  and for each point $x\in J$ there is a unique point $h(x)\in J$  with $V\big(h(x)\big)=V(x)$, and  $xh(x)<0$ for $x\ne 0$. We easily see that the function $h$ is a smooth involution.
The origin in $\R^4$ is Lyapunov stable for \eqref{cf} if and only if for $x\ne 0$ in a neighborhood of $0$ we have
\begin{equation}\label{pc}
\frac{1}{V(x)}=\frac12\left(\frac1x-\frac{1}{h(x)}\right)^2,
\end{equation}
see formula (4.3) in \cite{Z1}. In particular,  if  $f$ is even then so is $V(x)$ and $h(x)=-x$,  formula \eqref{pc} is equivalent to $V(x)=x^2/2$ and  we have
stability  if and only if $f$ is constant in a neighborhood of~$0$. This particular case was studied in \cite{ZB}
with a different approach.
\begin{figure}
\begin{center}
\includegraphics[width=0.9\textwidth]{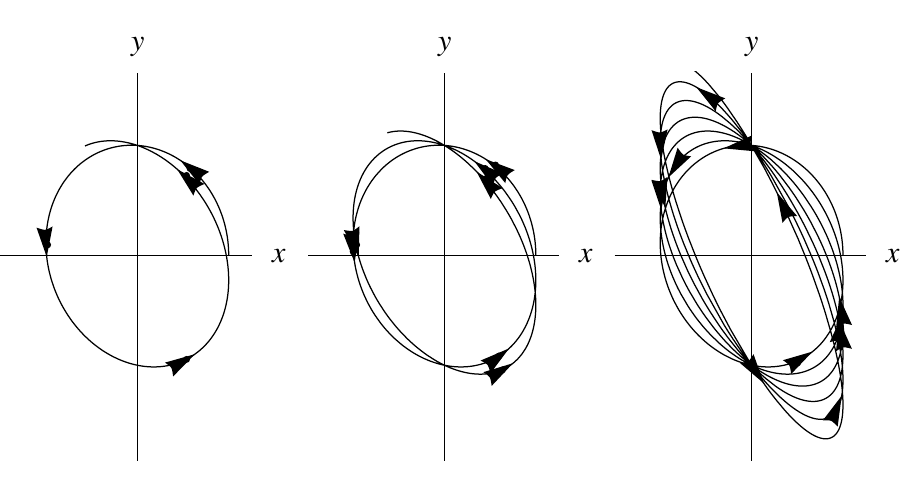}
\end{center}
\caption{Projection on the $x,y$-plane of a solution to \eqref{cf}.}
\label{illimitato}
\end{figure}

In Figure~\ref{illimitato} you can see the projection on the $x,y$-plane of a solution to \eqref{cf} for
$f(x)=1+x^2$. In this case, the origin is an unstable equilibrium for \eqref{cf}. The initial condition for the solution in Figure~\ref{illimitato} is $(x(0),\dot x(0),y(0),\dot y(0))=(0.4,0,0,0.5)$, on the left 
 $t\in[0,8]$, in the central picture $t\in[0,14]$, and on the right $t\in[0,38]$. It is an unbounded motion
 (see \cite{Z1} for details).

\section{Functional-differential equations with involutions}\label{odes}
Consider the following problem which involves the involution \eqref{h-upperrectangularhyperbola} on the interval $(-1,+\infty)$, the parameter $a\in\R$, and the initial datum $y_0\in\R$ at $t=0$
 \begin{equation}\label{problemwithh}
\begin{cases}y'(t)=a\,y\left(h(t)\right)\\ y(0)=y_0\end{cases},\qquad 
h(t)=-\frac{t}{1+t},\qquad t>-1\,. 
\end{equation}
If $y(t)$ is a $C^1$ solution then it is $C^2$. By differentiation we get
 \begin{equation*}
y''(t)=a\,h'(t)\,y'\left(h(t)\right)=
-\frac{a^2}{(1+t)^2}\,y\left(h\bigl(h(t)\bigr)\right)=-\frac{a^2}{(1+t)^2}\,y(t)\,. 
\end{equation*}
So  \eqref{problemwithh} is equivalent to the ordinary Cauchy problem
\begin{equation}\label{secondorder}
\begin{cases}y''(t)=\displaystyle{-\frac{a^2}{(1+t)^2}}\;y(t)\\ y(0)=y_0\\ y'(0)=a\,y_0\end{cases},\qquad 
 t>-1\,. 
\end{equation}
The solution is defined on the whole $(-1,\infty)$. For $|a|>1/2$:
\begin{equation*}
 y(t)=y_0\sqrt{1+t}\left(\cos\left(c\ln(1+t)\right)+\textstyle{\frac{2a-1}{2c}}
\sin\left(c\ln(1+t)\right)\right),\end{equation*}
where $c:=\sqrt{4a^2-1}/2$. While for $a=1/2$:
\begin{equation*}
y(t)=y_0\sqrt{1+t}\,.
\end{equation*}
For $a=-1/2$:
\begin{equation*}y(t)=y_0\sqrt{1+t}\,\left(1-\ln (1+t)\right).\end{equation*}
For $|a|<1/2$:\begin{equation*}
y(t)=\frac{y_0}{2\,b}\,(1+t)^{\frac{1-b}{2}}\Bigl(b+1-2a+(b-1+2a)\,(1+t)^b\Bigr),
\end{equation*}
where $b:=\sqrt{1-4a^2}$. This is just an example  of  functional-differential equations of Carleman type, a general theory is treated in  Chapter VIII
of Przeworska-Rolewicz \cite{PR} where references by other authors are quoted.
Equations with involutions are also studied in  \cite{BT}, \cite{CI}, \cite{DI}, \cite{SW}, \cite{W}, \cite{W1}, \cite{W2}, \cite{WW}.


\subsection*{Acknowledgements}
This research was partly supported by the PRIN ``Equazioni differenziali ordinarie e applicazioni''.

\end{document}